\documentclass[review]{elsarticle}

\usepackage{lineno,hyperref}
\modulolinenumbers[5]

\journal{Journal of \LaTeX\ Templates}

\usepackage{hyperref}
\usepackage{amsmath}
\usepackage{amsfonts}
\usepackage{amssymb}
\usepackage{amsthm}
\usepackage[all]{xy}
\usepackage{float}
\usepackage{fullpage}

\theoremstyle{plain}
\newtheorem{theorem}{Theorem}

\newtheorem{corollary}{Corollary}
\newtheorem{lemma}{Lemma}
\newtheorem{proposition}{Proposition}

\theoremstyle{definition}

\newtheorem{remark}{Remark}

\begin{document}
	
	\begin{frontmatter}
		
		\title{On Belyi's Theorems in Positive  Characteristic}
		
		
		
		\author{Nurdag\"ul Anbar}
		\address{Sabanc\i University, MDBF, Orhanlı, Tuzla, 34956, \.Istanbul, Turkey \\ Johannes Kepler University, Altenberger St. 69, 4040,
			Linz, Austria }
		
		
		\author{Seher Tutdere \corref{mycorrespondingauthor}}
		\cortext[mycorrespondingauthor]{Corresponding author}
		\ead{stutdere@gmail.com>}
		\address{Gebze Technical University, Department of Mathematics, 41400 Gebze, Kocaeli, Turkey}

		
		\begin{abstract}
		There are two types of Belyi's Theorem for curves defined over finite fields of characteristic $p$, namely the  Wild and the Tame $p$-Belyi Theorems. In this paper, we discuss them in the language of function fields. We provide a self-contained proof for the Wild $p$-Belyi Theorem for any prime $p$ and the Tame $2$-Belyi Theorem.
		
		\end{abstract}

		
	\end{frontmatter}
	

\section{Introduction}\label{introduction}

Let $\mathcal{X}$ be a connected, smooth, projective curve defined over the field of algebraic numbers $\bar{\mathbb{Q}}$. The main theorem of Belyi states that there exists a morphism
$f$ from $\mathcal{X}$ to the projective line  $\mathbb{P}^{1}$ such that the branch points of $f$ lie in the set $\{0,1,\infty\}$. The morphism $f$ satisfying this property is called a \textit{Belyi map} for $\mathcal{X}$. Belyi gave two
elementary proofs for his theorem, see \cite{bel1, bel2}.
In fact, the converse of the statement also holds, and was known before Belyi's result \cite{weil}. In other words, $\mathcal{X}$ is a curve defined over $\bar{\mathbb{Q}}$ if and only if there exists a morphism $f:\mathcal{X} \to \mathbb{P}^{1}$ whose branch points of $f$ lie in the set $\{0,1,\infty\}$. The connection with different areas of mathematics, such as the arithmetic and modularity  of  elliptic curves, $\mathrm{ABC}$ conjecture and moduli spaces of pointed curves, makes Belyi's statement more interesting, for details see the excellent paper \cite{gold1} and references therein.

In this paper we investigate Belyi's Theorem in positive characteristic $p$. We denote by $\mathbb F_q$ the finite field with $q$ elements, where $q$ is a power of a prime $p$, and by $\bar{\mathbb F}_p$ the algebraic closure of $\mathbb F_q$. The dichotomy of wild and tame ramification in positive characteristic leads to two types of Belyi's Theorem as follows.

\begin{theorem}[Wild $p$-Belyi Theorem]\label{wild-curve}
	Let $\mathcal{X}$ be a connected, smooth, projective curve defined over $\mathbb F_q$. Then there exists a morphism $\phi: \mathcal{X} \to \mathbb{P}^{1}$ over $\mathbb F_q$ which  admits at most one branch point.
\end{theorem}

\begin{theorem}[Tame $p$-Belyi Theorem]\label{tame-curve}
	Let $\mathcal{X}$ be a connected, smooth, projective curve defined over $\bar{\mathbb F}_p$. Then there exists a tamely ramified morphism $\phi: \mathcal{X} \to \mathbb{P}^{1}$ admitting  at most three branch points.
\end{theorem}

\noindent To the best of our knowledge, a first proof of Theorem \ref{tame-curve} for odd characteristic  is  given in \cite{said}. Moreover, in \cite{gold1}, the proofs of Theorem \ref{wild-curve} for any positive characteristic and Theorem \ref{tame-curve} for odd characteristic are given by using the results of \cite{katz, zap} and \cite{fulton}, respectively. The conjecture of the Tame $p$-Belyi Theorem for even characteristic has been recently proved in \cite{sy}.

It is a well-known fact that the theory of algebraic curves and the theory of algebraic function fields are equivalent \cite{har,nx2}. As a consequence of this equivalence, we here discuss Belyi's theorems in positive characteristic in the language of function fields. In fact, this significantly simplifies the proof of the Tame $2$-Belyi Theorem given in \cite{sy}.

The paper is organized as follows. In Section \ref{pre} we  fix  notations and give some basic facts regarding function fields. In Section \ref{sec:wild} we give a self-contained proof for the Wild $p$-Belyi Theorem. In Section \ref{sec:tame} we discuss the Tame $p$-Belyi Theorem. In particular, for $p>2$ we give the Tame $p$-Belyi Theorem by using the result \cite{fulton} and give a self contained proof of the Tame $2$-Belyi Theorem.  

\section{Preliminaries}\label{pre}

For the notations and  well-known facts, as a general reference, we refer to \cite{hkt,sti}.
Let $F$ be a function field over $\mathbb{F}$, where $\mathbb{F}=\mathbb F_q$ or $\mathbb{F}=\bar{\mathbb F}_p$, and let ${F}^{\prime}/F$ be a finite separable extension of function fields. We write ${P}^{\prime}|P$ for a place ${P}^{\prime}$ of $ F^{\prime}$ lying over a place $P$ of $F$, and denote by $e(P^{\prime}| P)$ the ramification index of ${P}^{\prime}|P$.
Recall that when the ramification index   $e(P^{\prime}| P)>1$,  it is said that $P^{\prime}|P$ is ramified. Moreover,  if  the characteristic $p$ of $\mathbb{F}$ does not divide  $e(P^{\prime}| P)$, then it is called \emph{tamely} ramified; otherwise it is called \emph{wildly} ramified. We call $F^{\prime }/F$ a tame extension if there is no wild ramification. For a rational function field $\mathbb{F}(y)$ and $\alpha\in \mathbb{F}$, we denote by $(y=\alpha)$ and $(y=\infty)$ the places corresponding to the zero and the pole of $y-\alpha$, respectively. 

We can state Belyi's theorems given in Theorems \ref{wild-curve} and \ref{tame-curve} in the language of function fields as  follows.

\begin{theorem}[Wild $p$-Belyi Theorem]\label{wild:ff}
	Let $F$ be a function field over $\mathbb F_q$. Then there exists a rational subfield $\mathbb F_q(y)$ of $F$ such that there exists at most one ramified place of $\mathbb F_q(y)$, namely $(y=\infty)$, in $F/\mathbb F_q(y)$.
\end{theorem}
\begin{theorem}[Tame $p$-Belyi Theorem]\label{tame:ff}
	Let $F$ be a function field over $\bar{\mathbb{F}}_p$. Then there exists a rational subfield $\bar{\mathbb F}_p(y)$ of $F$ such that $F/\bar{\mathbb F}_p(y)$ is a tame extension, and there exist at most three ramified places of $\bar{\mathbb F}_p(y)$ in $F/\bar{\mathbb F}_p(y)$ lying in the set $\{(y=0), (y=1), (y=\infty)\}$.
\end{theorem}

For the convenience of reader, we  now fix some notations. We denote by
\begin{itemize}
	\item[] $\mathbb{P}_F$ \hspace{2.25cm} the set of all places of $F/\mathbb{F}$,
	\item[] $[F^{\prime}:F]$ \hspace{1.6cm} the extension degree of $F^{\prime}/F$,
	\item[] $f(P^{\prime}| P)$ \hspace{1.5cm} the relative degree of $P^{\prime}| P$,
	\item[] $d(P^{\prime}| P)$ \hspace{1.5cm} the different exponent of $P^{\prime}|P$,
	\item[] $v_{P}$ \hspace{2.3cm} the valuation  of $F$ associated to the place $P$,
	\item[] $(z)_{\infty}$ (resp., $(z)_0$) \hspace{0.1cm} the pole divisor (resp., the zero divisor) of a nonzero element $z \in F$,
	\item[] $\mathcal  L (A)$ \hspace{2cm}  the Riemann-Roch space associated to a divisor $A$,
	\item[] $\ell(A)$ \hspace{2.1cm} the $\mathbb F$-dimension of $\mathcal L(A)$,
	\item[] $\mathrm{supp}(A)$ \hspace{1.5cm} the support of $A$, i.e., the set of places $P\in \mathbb{P}_F$ for which $v_P(A)\neq 0$.
\end{itemize}

Dedekind's Different Theorem \cite[Theorem 3.5.1]{sti} states that $d(P^{\prime}| P)\geq e(P^{\prime}| P)-1$, and the equality holds if and only if $P^{\prime}|P$ is tame. Furthermore, $P^{\prime}| P$ is ramified if and only if $d(P^{\prime}| P)>0$. By the Fundamental Equality \cite[Theorem 3.1.11]{sti}, we have $\sum{e(P^{\prime}| P)f(P^{\prime}|P)}=[F^{\prime}:F]$, where $P^{\prime}$ ranges over the places of $F^{\prime}$ lying over $P$.

One of the main tools in our proof of the Tame $2$-Belyi Theorem is the Strong Approximation Theorem \cite[Theorem 1.6.5]{sti}, and hence we state it for the sake of the reader.

\begin{lemma}\label{sat}
	Let $S\subset \mathbb{P}_F$ be a proper subset, and $P_1, \ldots P_r \in S$. For given $x_1, \ldots x_r \in F$ and $n_1, \ldots , n_r \in \mathbb{Z}$, there exists $x\in F$ such that 
	\begin{align*}
	v_{P_i}(x-x_i)=n_i \;\text{ for } i=1, \ldots r, \quad \text{and} \quad v_P(x)\geq 0 \; \text{ for all } P\in S\setminus \lbrace P_1, \ldots P_r  \rbrace \ .
	\end{align*}
\end{lemma}

\begin{corollary}\label{rem:sat}
	Let $D=\sum n_iP_i$, $n_i\geq 0$, be a positive divisor. Then the Strong Approximation Theorem implies the existence of $x\in F$ with $D\leq (x)_0$ and $(x)_{\infty}=nP$ for some place $P\not \in \mathrm{supp}(D)$ and $n\in \mathbb{N}$. 
\end{corollary}

In fact, we obtain a stronger conclusion by using the Riemann-Roch Theorem \cite[Theorem 1.5.15]{sti}.

\begin{lemma}\label{lem:rrt}
	Let $D=\sum n_iP_i$, $n_i\geq 0$, a divisor of degree $d$. Then for any $n\geq 2g+d$ there exists $x\in F$ with  $D\leq (x)_0$ and $(x)_{\infty}=nP$ for some place $P\not \in \mathrm{supp}(D)$. 
\end{lemma}

\begin{proof}
	Consider the Riemann-Roch spaces $\mathcal{L}(nP-D)$ and $\mathcal{L}((n-1)P-D)$. Since $n\geq 2g+d$, by the Riemann-Roch Theorem we have $\ell (nP-D)>\ell((n-1)P-D)$. Therefore, there exists $x\in \mathcal{L}(nP-D)\setminus \mathcal{L}((n-1)P-D)$, which is an element with desired properties.
\end{proof}

\subsection{Ramification in  the rational function field extensions}\label{rational}

Let $\mathbb F_q(x)/\mathbb F_q(t)$ be the rational function field extension given by the equation $t=\frac{g(x)}{h(x)}$ for some relatively prime polynomials $g(T),h(T)\in \mathbb F_q[T]$ such that not both $g,h$ lie in $\mathbb F_q[T^p]$. Without loss of generality, we assume that $\mathrm{deg}(g)> \mathrm{deg}(h)$; otherwise we consider the extension $\mathbb F_q(x)/\mathbb F_q(1/(t+\alpha))$ for some proper $\alpha \in \mathbb F_q$. Let $P$ be a place of $\mathbb F_q(x)$ of degree $r$, which is not the pole of $x$ or a zero of $h(x)$. Consider the constant field extensions $\mathbb F_q(t)\mathbb F_{q^r}\subseteq  \mathbb F_q(x)\mathbb F_{q^r}$, see Figure \ref{fig:constant}.
We have $[\mathbb F_q(x)\mathbb F_{q^r}:\mathbb F_q(x)]=[\mathbb F_q(t)\mathbb F_{q^r}:\mathbb F_q(t)]=r$ and the extension $\mathbb F_q(x)\mathbb F_{q^r}/\mathbb F_q(t)\mathbb F_{q^r}$ is defined by the same equation $t=\frac{g(x)}{h(x)}$. Note that any place $P^{\prime}\in \mathbb{P}_{\mathbb F_q(x) \mathbb F_{q^r}}$ lying over $P$ is of degree one, i.e., $P^{\prime}=(x=\alpha)$ for some $\alpha \in \mathbb F_{q^r}$. We set $Q^{\prime}=P^{\prime}\cap \mathbb F_q(t)\mathbb F_{q^r}$ and $Q=P^{\prime}\cap \mathbb F_q(t)$. Then $Q^{\prime}=(t=\beta)$, where $\beta=g(\alpha)/h(\alpha)$. Since there is no ramification in a constant field extension \cite[Theorem 3.6.3]{sti}, by the transitivity of ramification indices, we have $e(P| Q)=e(P^{\prime}|Q^{\prime})$. Write $g(T)-\beta h(T)=(T-\alpha)^mr(T)$ for some positive integer $m$ and $r\in \mathbb{F}[T]$ such that $r(\alpha)\neq 0$. We then have
\begin{align}\label{eq:rat}
e(P^{\prime}|Q^{\prime})=v_{P^{\prime}}(t-\beta)=v_{P^{\prime}}(g(x)-\beta h(x))=m \ .
\end{align}
In particular, Equation \eqref{eq:rat} implies that $P|Q$ is ramified if and only if $g(T)-\beta h(T)$ has multiple roots in $\bar{\mathbb F}_p$.
Note that any zero of $h(x)$ is a pole of $t$. Let $h(T)=\prod p_i(T)^{e_{p_i}}$ be the  factorization of $h(T)$ in $\mathbb{F}_q[T]$, where $p_i(T)$'s are distinct irreducible polynomials and $e_{p_i}\geq 1$. We denote by $P_i$ the place of $\mathbb F_q(x)$ corresponding to $p_i(x)$. Then the conorm of $(t=\infty)$ with respect to $\mathbb{F}_q(x)/\mathbb{F}_q(t)$ is given by
\begin{align*}
\mathrm{Con}_{\mathbb{F}_q(x)/\mathbb{F}_q(t)}\left((t=\infty)\right)=e((x=\infty)|(t=\infty))(x=\infty)+\sum e(P_i|(t=\infty))P_i \ .
\end{align*}
with
\begin{align*}
e((x=\infty)|(t=\infty))=\mathrm{deg}(g(T))-\mathrm{deg}(h(T)) \quad \text{and} \quad e(P_i|(t=\infty))= e_{p_i} \ .
\end{align*}
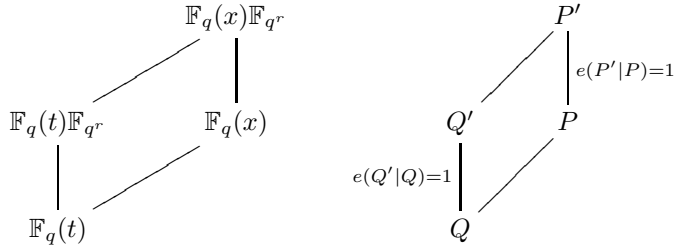
\begin{figure}[!ht]
	\begin{center}{
			\xymatrix{
				&&&& \mathbb F_q(x)\mathbb F_{q^r}&&&{P}^{\prime}\\
				&&&\mathbb F_q(t)\mathbb F_{q^r}\ar@{-}[ur]& \mathbb F_q(x)\ar@{-}[u]&&{Q}^{\prime}\ar@{-}[ur]&P\ar@{-}[u]_{e({P}^{\prime}|P)=1}\\
				&&&\mathbb F_q(t)\ar@{-}[u]\ar@{-}[ur]&&& Q\ar@{-}[u]^{e({Q}^{\prime}|Q)=1} \ar@{-}[ur]&}}
	\end{center}
	\caption{Constant field extensions of rational function fields}\label{fig:constant}
\end{figure}

We finish this section with the following lemma, which is required for the proofs of both $p$-Belyi theorems in the subsequent sections.

\begin{lemma}\label{main} Let $\mathbb F_q(x)$ be a rational function field, and let $S=\{ P_{1},\cdots,P_{n}\}$ be a finite set of places of $\mathbb F_q(x)$ with $P_{i}\not\in\{(x=0),(x=\infty)\}$ for all $i=1,\cdots,n$. Then there exists a subfield $\mathbb F_q(t)$ of $\mathbb F_q(x)$ with the following properties.
	\begin{itemize}
		\item [(i)] The extension $\mathbb F_q(x)/\mathbb F_q(t)$ is tame,
		\item [(ii)] $P_{i}$ lies over $(t=0)$ for all $i=1,\cdots,n$, and
		\item [(iii)] $(t=1)$ and $(t=\infty)$ are the only ramified places
		of $\mathbb F_q(t)$ in $\mathbb F_q(x)/\mathbb F_q(t)$.
	\end{itemize}
\end{lemma}
\begin{proof}
	We denote by $r_{i}$ the degree of $P_{i}$ for $i=1,\cdots,n$, and
	set $r=\mathrm{lcm}(r_{1},\cdots,r_{n})$, where $\mathrm{lcm}$ is the least common multiple. Consider the
	subfield $\mathbb F_q(t)$ of $\mathbb F_q(x)$ given by the equation $t=1-x^{q^{r}-1}$. Then $\mathbb F_q(x)/\mathbb{F}_q(t)$ is an extension of degree $q^{r}-1$. Since $r$ is divisible by the degree of $P_{i}$, by above discussion on ramification in the rational function fields extension, all the places $P_{i}$'s lie over $(t=0)$. Furthermore, $(x=\infty)$ and $(x=0)$ are the only places lying over $(t=\infty)$ and $(t=1)$, respectively, with ramification indices $e((x=\infty)| (t=\infty))=e((x=0)| (t=1))=q^{r}-1$ (see Figure \ref{2}). As the polynomial $T^{q^{r}-1}+\beta$ has no multiple roots for any nonzero $\beta \in \bar{\mathbb F}_p$, these is no other ramification. In particular, $\mathbb F_q(x)/\mathbb F_q(t)$ is a tame extension.	
	
	\begin{figure}[!ht]
		\begin{center}{
				\xymatrix{
					&\mathbb F_q(x)&&P_1 & \ldots \ldots & P_n&(x=\infty)&(x=0)\\
					&\mathbb F_q(t)\ar@{-}[u]&&&(t=0)\ar@{-}[ul]_{e=1}\ar@{-}[ur]_{e=1}&&(t=\infty)\ar@{-}[u]_{e=q^r-1}&(t=1)\ar@{-}[u]_{e=q^r-1}&&&&
				}}
			\end{center}
			\caption{Ramification structure in $\mathbb F_q(x)/\mathbb F_q(t)$} \label{2}
		\end{figure}
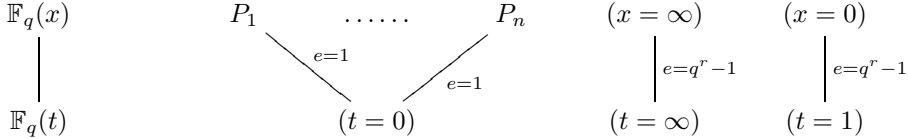 	
	\end{proof}

	\section{The Wild $p$-Belyi Theorem}   \label{sec:wild}

	In this section, we give a self-contained proof for the Wild $p$-Belyi Theorem for any positive characteristic $p$.

	\noindent \textbf{Proof of Theorem \ref{wild:ff}.} Let $x\in F$ be a separating element. Then there exist finitely many ramified places of $\mathbb{F}_q(x)$ in $F/\mathbb{F}_q(x)$. Assume that the ramified places lie in the set ${S}=\{(x=0),(x=\infty), P_{1},\cdots,P_{n}\}\subset \mathbb{P}_{\mathbb{F}_q(x)}$ for some $n\geq 1$. By Lemma \ref{main}, we can find an element $t\in \mathbb{F}_q(x)\subseteq F$ such that any ramified place of $F$ in $F/\mathbb F_q(t)$ lies over a place in the set $\{(t=0),(t=1),(t=\infty)\}$.
	
	We first consider the extension $\mathbb F_q(t)/\mathbb F_q(u)$ given by the equation $u=\frac{t^{p+1}+1}{t}$. The places $(t=0)$ and $(t=\infty)$ lie over $(u=\infty)$ with ramification indices $e((t=0) |(u=\infty))=1$ and $e((t=\infty)|(u=\infty))=p$ (see Figure \ref{3}). Hence, by the Fundamental Equality $(t=0)$ and $(t=\infty)$ are the only places lying over $(u=\infty)$. We have seen in Subsection \ref{rational} that there is no other ramification in $\mathbb F_q(t)/\mathbb F_q(u)$ if $f_{\beta}(T)=T^{p+1}-\beta T+1$ is a polynomial without multiple root for all $\beta\in\bar{\mathbb F}_p$. Suppose that $\alpha$ is a multiple root of $f_{\beta}(T)$ for some $\beta\in\bar{\mathbb F}_p$. Then $\alpha$ is also a root of $f'_{\beta}(T)= T^{p}-\beta$, and hence $\alpha$ is a $p$-th root of $\beta$. However, this means that $f_{\beta}(\alpha)=1$, which gives a contradiction. Moreover, the place $(t=1)$ lies over $(u=2)$. (Note that this is $(u=0)$ in characteristic $2$.)

Next, we consider the extension $\mathbb F_q(u)/\mathbb F_q(y)$ given by the equation $y=\frac{(u-2)^{p+1}+1}{u-2}$. Similarly, we can show that $(u=\infty)$ and $(u=2)$ are all places lying over $(y=\infty)$, and the  ramification occurs  only at $(y=\infty)$. Consequently, $(y=\infty)$ is the only ramified place in the extension $F/\mathbb F_q(y)$.

	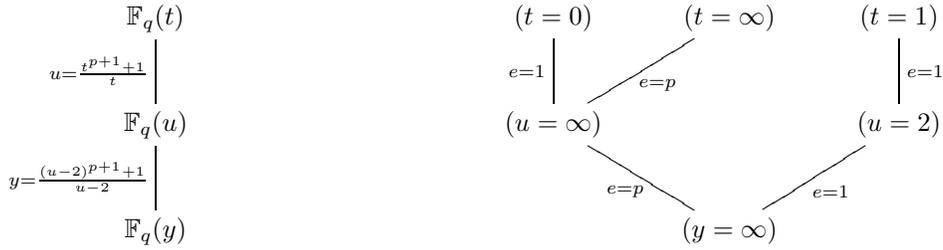
\begin{figure}[!ht]
		\begin{center}{
				\xymatrix{
					&&&\mathbb F_q(t)\ar@{-}[d]_{u=\frac{t^{p+1}+1}{t}}&&&&(t=0)&(t=\infty)&(t=1)&&&&\\
					&&&\mathbb F_q(u)\ar@{-}[d]_{y=\frac{(u-2)^{p+1}+1}{u-2}}&&&&(u=\infty)\ar@{-}[u]^{e=1}\ar@{-}[ur]_{e=p}&&(u=2)\ar@{-}[u]_{e=1}&&&& \\
					&&&\mathbb F_q(y)&&&&&(y=\infty)\ar@{-}[ur]_{e=1}\ar@{-}[ul]^{e=p}&&&&&}}
		\end{center}
		\caption{The Wild $p$-Belyi Theorem} \label{3}
	\end{figure}
	
	\hfill$\Box$\\[.5em]

	\begin{remark}\label{rem.wild}
		We note that in the proof of Theorem \ref{wild:ff} the ramified places in $\mathbb F_q(t)/\mathbb F_q(u)$ and $\mathbb F_q(u)/\mathbb F_q(y)$ have ramification indices $p$, i.e., they are wild, see Figure \ref{3}. It follows from the  Hurwitz Genus Formula \cite[Theorem 3.4.13]{sti} that both ramification have different exponents $2p$.
	\end{remark}

	\section{The Tame $p$-Belyi Theorem}\label{sec:tame}

	As mentioned in \cite{gold1}, a proof of Theorem \ref{tame:ff} for $p>2$ can be given as an application of the following  technical result of Fulton, which shows the existence of a tame rational subfield of a function field. 
	\begin{proposition} \cite[Proposition 8.1]{fulton} \label{fulton}
		If $F$ is a function field with constant field $\bar{\mathbb F}_p$ with $p>2$, then there exists a rational subfield $\bar{\mathbb F}_p(x)$ of $F$ such that $e(Q| P)=2$ or $1$ for any $Q\in\mathbb{P}_{F}$ and $P\in\mathbb{P}_{\bar{\mathbb F}_p(x)}$ with $Q| P$. 
	\end{proposition}
	Therefore, we  first prove the existence of a tame rational subfield of a function field $F$ over $\bar{\mathbb{F}}_p$ for $p=2$.  We will then give a proof of Theorem \ref{tame:ff}.
	
	\subsection{The Tame $2$-Belyi Theorem}\label{p=2}

	Throughout this subsection, we assume that $F$ is a function field over $\mathbb{F}=\bar{\mathbb{F}}_2$. An element $x\in F$ is called \textit{pseudo-tame at} $P\in \mathbb{P}_F$ if there exists $z\in F$ such that $x+z^4$ is tame at $P$. Moreover, we say that $x$ is a \textit{pseudo-tame element} of $F$ if $x$ is pseudo-tame at $P$ for all $P\in  \mathbb{P}_F$.

	\begin{lemma}\label{lem:aa} Let $x\in F$ and  $\Gamma$ be the projective general linear group over $F^4$.
		\begin{itemize}
			\item[(i)]  $x$ is pseudo-tame at $P$ if and only if the degree of any non-vanishing term in the Laurent series expansion of $x$ with smaller than $v_P(dx)+1$ is multiple of four.
			\item[(ii)] $x$ is pseudo-tame at $P$ if and only if $\gamma(x)$ is pseudo-tame at $P$ for any $\gamma \in \Gamma$.
		\end{itemize}
	\end{lemma}
	\begin{proof}
		\begin{itemize}
			\item[(i)] The proof is straightforward by the definition of being pseudo-tame.
			\item[(ii)] It is enough to observe that if $x$ is pseudo-tame at $P$, then $a^4x+b^4$ and $1/x$ are also pseudo-tame at $P$ by $(i)$.
		\end{itemize}
	\end{proof}
	
	For $x,y\in \mathcal{H}=F\setminus F^2$, we write $x=x_0^4+x_1^4y+x_2^4y^2+x_3^4y^3$ for some $x_0,x_1,x_2,x_3\in F$ and define
	\begin{align}\label{eq:a}
	a(x,y)= \frac{(x_1^2x_3^2+x_2^4)y}{x_3^4y^2+x_1^4}  \ .
	\end{align}
	The notion $a(x,y)$ is introduced in \cite{sy}. We can summarize the required properties of $a(x,y)$, which are given in Proposition 2.7 and Theorem 2.10 in \cite{sy}, as follows.
	\begin{lemma}\label{lem:a}
		\begin{itemize}
			\item[(i)] For any $x,y,t\in \mathcal{H}$
			\begin{align}\label{eq:cocycle}
			a(x,y)+a(y,t)+a(t,x)\equiv 0 \mod F^2 \ .
			\end{align}
			\item[(ii)]Let $a(x,y)\equiv a \mod F^2$ and $y$ be pseudo-tame at $P$. Then $x$ is pseudo-tame at $P$ if and only if $a$ is regular at $P$, i.e., there exists $\tilde{a}\in F$ with $\tilde{a}\equiv a \mod F^2$ and $v_P(\tilde{a})\geq 0$.
		\end{itemize}
	\end{lemma}

	\noindent One of the main tools to show the existence of a pseudo-tame element is \textit{Tsen's Theorem} stated as follows:\\
	A function field $F$ over $\bar{\mathbb{F}}_p$ is quasi-algebraically closed, i.e, any homogeneous polynomial over $F$ of $n$ variables whose degree is less than $n$ has a non-trivial solution.

	\begin{proposition}\label{pro:g}
		For any $x,a\in \mathcal{H}$, there exists $y\in \mathcal{H}$ such that $a(x,y)\equiv a \mod F^2$.
	\end{proposition}
	\begin{proof}
		Since $F=F^2\oplus xF^2$, there exists unique $b\in F$ such that $a\equiv b^2x \mod F^2$. Write $y=y_0^4+y_1^4x+y_2^4x^2+y_3^4x^3$. Note that by Equation \eqref{eq:cocycle}, $a(x,y)\equiv a(y,x) \mod F^2$, and hence
		\begin{align}\label{eq:mod}
		a(x,y)\equiv \frac{(y_1^2y_3^2+y_2^4)x}{y_3^4x^2+y_1^4}\equiv b^2x \mod F^2 \ .
		\end{align}
		This holds if and only if $b\equiv (y_1y_3+y_2^2)/(y_3^2x+y_1^2) \mod F^2$. By Tsen's Theorem, there exists an element $y\in F$ satisfying Equation \eqref{eq:mod}.
	\end{proof}
	
	We need the following two lemmata, which will be used in the proof of the existence of a pseudo-tame element.
	
	\begin{lemma}\label{z}
		Let $x\in \mathcal{H}$ and $P, Q \in \mathbb{P}_F\setminus \mathrm{supp}(x)_{\infty}$. Then there exists $z\in F$ such that $z$ has simple poles, $v_Q(z)\geq 0$, and $x+z^2$ is tame at $P$.
	\end{lemma}
	\begin{proof}
		Let $u\in F$ be a prime element at $P$, and $$x=a_0+a_1u+a_2u^2+\cdots$$ be the Laurent series expansion of $x$. Let $j$ be the integer such that $a_j$ is the first non-vanishing term in the expansion. Note that $j\geq 0$ as $P\not\in \mathrm{supp}(x)_{\infty}$. If $j$ is odd, then it is enough to choose a nonzero $z\in \mathbb{F}$. Suppose that $j=2n$. Then consider a divisor $R=R_1+\cdots +R_t$, where $t$ is sufficiently large, the $R_i$'s are pairwise distinct, and $P,Q\not \in \mathrm{supp}(R)$. By the Riemann Roch Theorem, there exists $z\in F$ such that
		\begin{align*}
		z\in \mathcal{L}(R-nP)\setminus \mathcal{L}(R-(n+1)P)\ .
		\end{align*}
		Then $z$ has simple poles, $v_Q(z)\geq 0$, and $v_P(z)=n$. There exists $\alpha \in \mathbb{F}$ with $v_P(x+\alpha z^2)>2n$. Then after finitely many steps we obtain an element satisfying the desired properties by Strict Triangle Inequality (\cite[Lemma 1.1.11]{sti}).
	\end{proof}
	
	\begin{lemma}\label{y}
		Let $R=R_1+\cdots +R_t$, and $P_1,\ldots,P_n,Q\in \mathbb{P}_F\setminus \mathrm{supp}(R)$, where $t$ is sufficiently large and the $R_i$'s are pairwise distinct. Then there exists $y\in F$ such that $(y)_{\infty}=R$, $P_i\not \in \mathrm{supp}(y)_0$, and $v_Q(y)\geq k$ for some positive integer $k$.
	\end{lemma}
	\begin{proof}
		By the Riemann Roch Theorem, there exist $z_j, x_i$ such that
		\begin{align*}
		z_j\in \mathcal{L}(R-kQ)\setminus \mathcal{L}(R-kQ-R_j) \quad \textrm{ and } \quad x_i\in \mathcal{L}(R-kQ)\setminus 
		\mathcal{L}(R-kQ-P_i)
		\end{align*}
		for all $i=1, \ldots, t$ and $j=1, \ldots n$.
		Note that $z_j,x_i$ have simple poles in the $\mathrm{supp}(R)$ with $v_{P_i}(x_i)=0$ and $v_{R_j}(z_j)=-1$. As $\mathbb{F}$ is algebraically closed, there exist $\alpha_j, \beta_i\in \mathbb{F}$ such that 
		\begin{align*}
		y=\sum_{j=1}^{t}\alpha_jz_j+\sum_{j=1}^{n} \beta_ix_i
		\end{align*}
		has the desired properties.
	\end{proof}  
	
	\begin{proposition}\label{pro:pt}
		Let $F$ be a function field over $\mathbb{F}=\bar{\mathbb{F}}_2$. Then there exists a pseudo-tame element $x\in F$.
	\end{proposition}
	
	\begin{proof}
		We first show the existence of $x_i,a_i\in F$ for $i=1,2$ such that $x_i$ is pseudo-tame, $a_i$ is regular at $P$ for all $P\in U_i$ with $\mathbb{P}_F =U_1\cup U_2$ and $a(x_1,x_2)\equiv  a_1+a_2 \mod F^2$. 
		
		Let $x_1\in F$ such that $(x_1)_{\infty}=(2n+1)Q$ for sufficiently large $n$ and $Q\in \mathbb{P}_F$. Moreover, we can suppose that $x_1$ has simple zeros. Otherwise, we can replace $x_1$ by $x_1+\alpha $ for some $\alpha \in \mathbb{F}$. Suppose $Q,P_1, \ldots, P_t $  are ramified places of $F$ in $F/\mathbb{F}(x_1)$. Let $z\in F$ such that
		\begin{itemize}
			\item $v_Q(z)\geq 0$,
			\item $z$ has simple poles such that $\mathrm{supp}((x_1)_0)\cap \mathrm{supp}((z)_{\infty})=\emptyset $, and
			\item $x_2=x_1+z^2$ is tame at $P_1, \ldots, P_t$.
		\end{itemize}
		Such an element exists by Lemma \ref{z}. We set $U_1=\mathbb{P}_F\setminus \lbrace P_1, \ldots, P_t \rbrace$ and $U_2=\lbrace P_1, \ldots, P_t \rbrace$.
		Observe that 
		\begin{align*}
		a=a(x_1,x_2)= \left( \frac{dz}{dx_1} \right)^2x_1 \ .
		\end{align*}
		As $v_Q(z)\geq 0$, we have $v_Q(a)\geq 0$. Also, it is easy to observe that $v_P(a)\geq 0$ for any $P\in \mathbb{P}_F\setminus \lbrace P_1, \ldots, P_t \rbrace \cup \mathrm{supp}((z)_{\infty})$ since $dx_1$ has  zeros  only at $P_i$ for $i=1, \ldots,t$ and $dz$ has only poles in $\mathrm{supp}((z)_{\infty})$. Say $\mathrm{supp}((z)_{\infty})=R_1+\cdots+R_k$, where the $R_i$'s are pairwise distinct places of $F$. As $k$ is sufficiently large, by Lemma \ref{y}, there exists $y\in \mathcal{L}(R_1+\cdots+R_k)$ such that
		\begin{itemize}
			\item $y$ has  zero at $Q$ of sufficiently large order,
			\item $v_{R_i}(y)=-1$ for all $i=1, \ldots,k$,
			\item $y$ has no zero at $P_1, \ldots, P_t$.
		\end{itemize}
		Set $u=\frac{1}{y}$, then we can write 
		\begin{align*}
		\frac{dz}{dx}=\alpha_{-2}\frac{1}{u^2}+\alpha_{-1}\frac{1}{u}+\alpha_0+\cdots
		\end{align*}  
		Note that the Laurent series expansion of $\frac{dx}{du}$ and $\frac{dz}{du}$ with respect to $u$ has only even powers of $u$, and hence we have $\alpha_{-1}=0$. Then
		
		\[ v_{R_i}\bigg (\frac{dz}{dx}+\alpha_{-2}\frac{1}{u^2}\bigg)\geq 0 \quad \textrm{ for all }  i=1,\ldots, k  \] 
		and $\frac{dz}{dx}+\alpha_{-2}\frac{1}{u^2}$ has zero at $Q$ of sufficiently large order. Set
		
		\[a_1=\bigg(\frac{dx}{dz}+\alpha_{-2}\frac{1}{u^2}\bigg)^2x \; \textrm{ and } \; a_2=\frac{\alpha_{-2}^2x}{u^4}\]
		so that $a=a_1+a_2$. Then $a_1$ is regular for all $P\in \mathbb{P}_F\setminus \lbrace P_1, \ldots, P_t \rbrace $. Furthermore, since $u$ has no pole at $P_i$, $a_2$ is regular at $P_i$ for all $i=1, \ldots,t$. 
		
		The rest of the proof is similar to the one given in \cite[Theorem 3.6]{sy}, but we give it here for the completeness. By Proposition \ref{pro:g}, for any $a_i$ there exists $y_i\in F$ such that $a(x_i,y_i)\equiv a_i \mod F^2$ for $i=1,2$. Then Equation \eqref{eq:cocycle} implies that $a(y_1,y_2)\equiv 0 \mod F^2$, i.e., $a(y_1,y_2)$ is regular at $P$ for all $P\in U_j$. By Lemma \ref{lem:a}/(ii), this shows that $y_i$ is pseudo-tame at $P$ for all $P\in U_j$ and $j=1,2$, $i=1,2$. In other words, $y_i$ is pseudo-tame at $P$ for all $P\in \mathbb{P}_F $ for $i=1,2$. 
	\end{proof}
	
	We fix a place $Q\in \mathbb{P}_F$, and set $R=\bigcup\limits_{{n\in \mathbb{N}}} \mathcal  L (nQ)$, i.e., $R$ is the
	the subring of $F$ consisting of all elements which have poles only at $Q$.

	\begin{lemma}\label{lem:P}
		Let $x\in R$. If $x$ is pseudo-tame at $Q$ with $-v_Q(dx)\geq 8g$, then there exists $z\in R$ such that $-v_Q(x+z^4)=-v_Q(dx)-1$.
	\end{lemma}
	
	\begin{proof}
		We set $2e=-v_Q(dx)$. Note that $-v_Q(x)\geq -v_Q(dx)-1$, and the equality holds only if $x$ is tame at $Q$. Suppose 
		that $-v_Q(x)> -v_Q(dx)-1\geq 8g -1$. Since $x$ is pseudo-tame at $Q$, $v_Q(x)=-4k$ for some integer $k \geq 2g$. By Lemma \ref{lem:rrt}, there exists $z_0\in F$ with $(z_0)_{\infty}=kQ$. Since $\mathbb{F}$ is algebraically closed, there exists $\alpha$ such that for $\tilde{x}=x+\alpha z_0^4$ we have $-v_Q(\tilde{x})<4k=-v_Q(x)$. Then the existence of $z$ follows after finitely many steps. 
		
	\end{proof}

	\begin{lemma}\label{lem:D}
		Let $D=\sum n_i P_i$, $n_i\geq 0$, be a divisor of degree $d$. Suppose that $Q\not\in \mathrm{supp}(D)$ and $d>2g$. Then for $a\in R$ there exists $x\in R$ such that $D\leq (x+a)_0$ and $(x)_{\infty}=nQ$ for some $n< d+2g$. 
	\end{lemma}
	
	\begin{proof} 
		By the Strong Approximation Theorem, there exists $x\in R$ such that $D\leq (x+a)_0$ and $(x)_\infty=nQ$, see Corollary \ref{rem:sat}.
		If $n\geq d+2g$, then there exists $z\in F$ such that $D\leq (z)_0$ and $(z)_{\infty}=nQ$ by Lemma \ref{lem:rrt}. There exists $\alpha\in \mathbb{F}$ such that $(x+\alpha z)_{\infty}=kQ$ with $k<n$. Note that for $\tilde{x}=x+\alpha z \in R$ we have $D\leq (\tilde{x}+a)_0$. Then the argument follows by induction.
		
	\end{proof}
	
	
	
	\begin{proposition}\label{pro=2}
		Let $F$ be a function field over $\mathbb{F}=\bar{\mathbb{F}}_2$. Then there exists $x\in F$ such that $F/\mathbb{F}(x)$ is tame.
	\end{proposition}
	\begin{proof}
		Let $x_0$ be a pseudo-tame element. As $F$ is the quotient field of $R$, we can  write $x_0=z_0/z_1$ for some $z_0,z_1 \in R$.  Set $y=x_0z_1^4=z_1^3z_0$. Note that $y\in R$ is pseudo-tame by Lemma \ref{lem:aa}/(ii). We can assume that $-v_Q(dy)\geq 8g$; otherwise we can replace $y$ by $z^4y$ for some proper $z\in R$. By Lemma \ref{lem:P}, we can assume that $-v_Q(dy)=-v_Q(y)-1=2e$.  Moreover, we can suppose that $y$ has simple zeros; otherwise replace $y$ by $y+\alpha$ for some proper $\alpha \in \mathbb{F}$. In other words, there exists a pseudo-tame element $y\in R$, which is tame at $Q$ and having simple zeros.

		Let $\mathcal{Z}$ be the set of zeros of $dy$. Observe that $y$ is pseudo-tame implies that $y^3$ is pseudo-tame. As $dy$ has finitely many zeros, there exists $z\in F$ such that $y^3+z^4$ is tame at $P$ for all $P\in \mathcal{Z}$. Moreover, by the Strong Approximation Theorem, we can assume that $z\in R$, i.e., we can assume that $y^3+z^4$ is a pseudo-tame element in $R$ which is tame at $P$ for all $P\in \mathcal{Z}$. We set $v_P(dy)=2m_P$, and define 
		\[ D=\sum_{P\in \mathcal{Z}} \bigg\lfloor \frac{m_P}{2}\bigg\rfloor Q \ .\] 
		As $\mathrm{deg}(dy)=2g-2$, we have 
		\[\sum_{P\in \mathcal{Z}} m_P=e+g-1, \textrm{ i.e., } \mathrm{deg}(D)\leq \frac{e+g-1}{2} \ . \]
		By Lemma \ref{lem:D}, we can also assume that $z\in R$ such that 
		\[(z)_0\geq D \textrm{ and } \mathrm{deg}(z)_{0}=\mathrm{deg}(z)_{\infty} \leq 2g+\frac{e+g-1}{2} \ .\] 
		We set $x=y^3+h^4$. Note that by construction $x\in R$ is pseudo-tame and tame at $P$ for all $P\in \mathcal{Z}$. Moreover, the Strict Triangle Inequality implies that 
		\[v_Q(x)=3v_P(Q)=-3(2e-1), \textrm{ i.e., $x$ is tame at }  Q.\] For $P\in \mathbb{P}_F \setminus \mathcal{Z}\cup\lbrace Q \rbrace $, we see that $v_P(dx)=2v_Q(y)=0$ or $2$ as $y$ has only simple zeros. Note that $x$ is unramified at $P$ if and only if $v_P(dx)=0$. Since $x$ is a pseudo-tame rational function, any term in the Laurent series expansion smaller than $v_P(dx)$ is multiple of $4$ by Lemma \ref{lem:aa}/(i). However, this implies that $v_P(dx)=0$, i.e., $x$ is tame at $P$.
		
	\end{proof}
	


	\noindent \textbf{Proof of Theorem \ref{tame:ff}.} 	We consider the subfield  $\bar{\mathbb F}_p(x)$ of $F$ given as in Propositions  \ref{fulton} and \ref{pro=2}, i.e.,  $F/\bar{\mathbb F}_p(x)$ is tame. Since $F/\bar{\mathbb F}_p(x)$ is a finite separable extension, there exist finitely many ramified places of $\bar{\mathbb F}_p(x)$ in $F/\bar{\mathbb F}_p(x)$. Suppose that all the ramified places of $\bar{\mathbb F}_p(x)$ are contained in the set $\{(x=0),(x=\infty), P_{1},\cdots,P_{n}\}$ for some $n\geq 1$. Note that any place of $\bar{\mathbb F}_p(x)$ is rational, i.e., $P_i$ is a place corresponding to $x-\alpha_i$ for some nonzero $\alpha_i \in \bar{\mathbb F}_p$. Let $r$ be a positive integer such that $\alpha_i^{q^r-1}-1=0$ for all $i=1,\ldots,n$. Then Lemma \ref{main} also holds for the extension $\bar{\mathbb F}_p(x)/\bar{\mathbb F}_p(t)$ defined by $t=1-x^{q^r-1}$. In other words, all places $P_1,\ldots , P_n$ lie over $(t=0)$. Moreover, $(x=0),(x=\infty)$ are the only ramified places in $\bar{\mathbb F}_p(x)/\bar{\mathbb F}_p(t)$, which are  totally ramified lying over $(t=1),(t=\infty)$, respectively. Then the proof follows from the fact that $\bar{\mathbb F}_p(x)/\bar{\mathbb F}_p(t)$ is tame.
	\hfill$\Box$\\[.5em]
	
	We note that the statement of the  Tame $p$-Belyi Theorem strictly holds if the genus of $F$ is positive. More precisely, we will see in Remark \ref{rem:tame} that there must be at least three (resp., two) ramified places in Theorem \ref{tame:ff} if $g(F)>0$ (resp., $g(F)=0$). That is, the places $(y=0), (y=1)$, and $(y=\infty)$ are all ramified in the Tame $p$-Belyi Theorem when $g(F)$ is positive.
	
	\begin{remark}\label{rem:tame} Let $F$ be a function field over $\bar{\mathbb F}_p$. Suppose that there exists a rational subfield $\bar{\mathbb F}_p(y)$ of $F$ such that
		$F/\bar{\mathbb F}_p(y)$ is tame of degree $n$. Let $Q_1, \ldots , Q_k$ be all ramified places of $\bar{\mathbb F}_p(y)$ in
		$F/\bar{\mathbb F}_p(y)$. We denote by $N_{Q_i}$ the number of places of $F$ lying over $Q_i$ for $i=1,\ldots, k$. Then by Dedekind's Different Theorem the degree of the ramification divisor of $F/\bar{\mathbb F}_p(y)$ is given as follows.
		\begin{eqnarray}\label{eq:ddt}
		\mathrm{deg}\left(\mathrm{Diff}(F/\bar{\mathbb F}_p(y))\right)
		\nonumber 	&=&\sum_{i=1}^{k} \sum_{P\in\mathbb{P}_{F},P|Q_{i}}{(e(P|Q_{i})-1)} \\
		\nonumber 	&=&\sum_{i=1}^{k} \sum_{P\in\mathbb{P}_{F},P|Q_{i}}{e(P|Q_{i})} -\sum_{i=1}^{k}N_{Q_i}\\
		&=&kn-\sum_{i=1}^{k}N_{Q_i}
		\end{eqnarray}
		Note that we use the Fundamental Equality in the last equality. By the Hurwitz genus formula, we also have 
		\begin{align}\label{eq:hgf}
		\mathrm{deg}\left(\mathrm{Diff}(F/\bar{\mathbb F}_p(y))\right)=2n+2g(F)-2 \ .
		\end{align}
		Equation \eqref{eq:ddt} and \eqref{eq:hgf} implies that $k\geq 2$. The case $k=2$ holds only if $g(F)=0$ and the places $Q_1 , Q_2$ are totally ramified. 
	\end{remark}
	
	\begin{remark}\label{rem:wild} Since ramification does not change under the constant field extension, we conclude from Remark \ref{rem:tame} that there must be a wild ramification in Theorem \ref{wild:ff} as noticed in Remark \ref{rem.wild}. Hence, it is called the Wild $p$-Belyi Theorem.
	\end{remark}

	\section*{Acknowledgment}

	The authors would like to thank Prof. Dr. Henning Stichtenoth and Prof. Dr. Jaap Top for their kind help and helpful discussions, which improve the manuscript considerably.
	Nurdag\"{u}l Anbar is supported by the Austrian Science Fund (FWF): Project F5505--N26 and Project F5511--N26, which is a part of the Special Research Program ``Quasi-Monte Carlo Methods: Theory and Applications".

\end{document}